\documentclass[oneside,english]{amsart}
\usepackage[T1]{fontenc}
\usepackage[latin9]{inputenc}
\usepackage{float}
\usepackage{amstext}
\usepackage{amsthm}
\usepackage{amssymb}
\usepackage{graphicx}

\makeatletter
\numberwithin{equation}{section}
\theoremstyle{plain}
\newtheorem{thm}{\protect\theoremname}
  \theoremstyle{plain}
  \newtheorem{cor}{\protect\corollaryname}
  \theoremstyle{remark}
  \newtheorem{rem}{\protect\remarkname}
  \theoremstyle{plain}
  \newtheorem*{thm*}{\protect\theoremname}
  \theoremstyle{plain}
  \newtheorem{lem}{\protect\lemmaname}

\date{}
\usepackage{ifpdf}
\ifpdf

\IfFileExists{lmodern.sty}
 {\usepackage{lmodern}}{}

\fi 
\usepackage[a4paper, hmargin=2.5cm, vmargin={3.2cm, 3.2cm}]{geometry}

\author{Avner Kiro, Yotam Smilansky, Uzy Smilansky}
\address{Avner Kiro, Raymond and Beverly Sackler School of Mathematical Sciences, Tel Aviv University, Tel Aviv 69978, Israel}
\email{avnerefrak@mail.tau.ac.il}
\address{Yotam Smilansky, Raymond and Beverly Sackler School of Mathematical Sciences, Tel Aviv University, Tel Aviv 69978, Israel}
\email{yotam.smilansky@mail.huji.ac.il}
\address{Uzy Smilansky, Department of Physics of Complex Systems, The Weizmann Institute of Science, Rehovot 76100, Israel}
\email{uzy.smilansky@weizmann.ac.il}

\makeatother

\usepackage{babel}
  \providecommand{\lemmaname}{Lemma}
  \providecommand{\remarkname}{Remark}
  \providecommand{\theoremname}{Theorem}
\providecommand{\corollaryname}{Corollary}
\providecommand{\theoremname}{Theorem}

\begin{document}
\global\long\def\adj{{\rm adj}}

\global\long\def\tr{{\rm tr}}

\global\long\def\span{{\rm span}}

\global\long\def\diag{{\rm diag}}

\global\long\def\vol{{\rm vol}}

\title{The Distribution of Path Lengths On Directed Weighted Graphs}
\begin{abstract}
We consider directed weighted graphs and define various families of
path counting functions. Our main results are explicit formulas for
the main term of the asymptotic growth rate of these counting functions,
under some irrationality assumptions on the lengths of all closed
orbits on the graph. In addition we assign transition probabilities
to such graphs and compute statistics of the corresponding random
walks. Some examples and applications are reviewed.
\end{abstract}

\maketitle

\section{Introduction and Main Results}

Questions regarding the distribution of path lengths on directed weighted
graphs are encountered in various fields of study of mathematics and
physics. They arise naturally in dynamics and the study of closed
orbits of suspensions of shifts of finite type, see among others \cite{key-7},
\cite{key-213241}, \cite{key-143513125} and \cite{key-2}, and the
more recent \cite{key-4} and \cite{key-1-1}.

Our motivations for counting paths on weighted graphs are diverse.
The second author's motivation originates in the study of a model
of mathematical quasicrystals which we call multiscale substitution
tilings, and in the study of equidistribution of what is known as
Kakutani's partitions, first described in \cite{key-3213}. The connection
to problems concerning path counting on weighted graphs is introduced
in subsection \ref{subsec:Multiscale-substitution}. The third author\textquoteright s
motivation is rooted in theoretical physics, specifically in the spectral
properties of the Schr\"odinger operator for systems which are chaotic
in the classical limit and for metric graphs \cite{key-3}. Of particular
relevance are studies of the distribution of delay (transit) times
through chaotic scatterers such as e.g., the scattering of ultra-short
electromagnetic pulses by complex molecules, or traversing networks
of transmission lines \cite{key-1} and \cite{key-345}.

\subsection{Counting paths in graphs}

Let $G=\left(\mathcal{V},\mathcal{E},l\right)$ be a directed weighted
metric graph with a set $\mathcal{V}=\left\{ 1,...,n\right\} $ of
vertices and a set $\mathcal{E}$ of edges. A positive weight is assigned
to each edge $\alpha\in\mathcal{E}$, and we think of this weight
as the length of $\alpha$. For a path $\gamma$ connecting two vertices
in $\mathcal{V}$, the path metric $l$ is defined to be the sum of
the weights of the edges in $\gamma$. When considering paths which
do not necessarily terminate at a vertex of $G$, the path metric
is defined by $l\left(\gamma\right)=a$ if the path $\gamma$ is isometric
to $\left[0,a\right]\subset\mathbb{R}$. Throughout this paper $G$
is assumed to be a strongly connected multigraph, that is a graph
which admits a path from every vertex $i\in\mathcal{V}$ to every
vertex $j\in\mathcal{V}$, and loops and multiple edges are allowed.

We say that $G$ is a graph of incommensurable orbits, or \textbf{incommensurable}
for short, if there exist at least two closed paths in $G$ of lengths
$a,b$ such that $a\not\in\mathbb{Q}b$. This irrationality condition
on the lengths of the edges is equivalent to the set of lengths of
all closed orbits in $G$ not being a uniformly discrete subset of
$\mathbb{R}$. Indeed, if there exist two closed paths in $G$ of
lengths $a,b$ such that $a\not\in\mathbb{Q}b$, then by Dirichlet's
approximation theorem for every $\varepsilon>0$ there exist $p,q\in\mathbb{N}$
such that $\left|aq-pb\right|<\varepsilon$, and so the set of lengths
of closed orbits in $G$ is not uniformly discrete. Conversely, if
the set of lengths of closed orbits is rationally dependent, then
the finiteness of the graph implies that there is a finite set $L$
of lengths for which the length of any closed orbit in $G$ is a linear
combination with integer coefficients of elements in $L$. It follows
that the set of lengths of closed orbits in $G$ is uniformly discrete.

Let $i,j\in\mathcal{V}$ be a pair of vertices in $G$, and assume
that there are $k\geq0$ edges $\alpha_{1},\ldots,\alpha_{k}$ from
$i$ to $j$. The matrix valued function $M:\mathbb{C}\rightarrow M_{n}\left(\mathbb{C}\right)$,
which we call the graph matrix function of $G$, is defined by
\[
M_{ij}\left(s\right)=e^{-s\cdot l\left(\alpha_{1}\right)}+\cdots+e^{-s\cdot l\left(\alpha_{k}\right)}
\]
and $M_{ij}\left(s\right)=0$ if $i$ is not connected to $j$ by
an edge. Note that the restriction of $M$ to $\mathbb{R}$ is real
valued.
\begin{thm}
\label{main result 1}Let $G$ be a strongly connected incommensurable
graph. There exist a positive constant $\lambda$ and a matrix $Q\in M_{n}\left(\mathbb{R}\right)$
with positive entries such that

$\left(i\right)$ The number of paths from $i\in\mathcal{V}$ to $j\in\mathcal{V}$
of length at most $x$ grows as
\[
\frac{1}{\lambda}Q_{ij}e^{\lambda x}+o\left(e^{\lambda x}\right),\quad x\rightarrow\infty.
\]

$\left(ii\right)$ Let $\alpha\in\mathcal{E}$ be an edge in $G$
which originates in vertex $j\in\mathcal{V}$. The number of paths
of length exactly $x$ from some vertex $i$ to a point on the edge
$\alpha$ grows as
\[
\frac{1-e^{-l\left(\alpha\right)\lambda}}{\lambda}Q_{ij}e^{\lambda x}+o\left(e^{\lambda x}\right),\quad x\rightarrow\infty.
\]
The constant $\lambda$ is the maximal real value for which the spectral
radius of $M$ is equal to $1$, and
\[
Q=Q\left(M\left(\lambda\right)\right)=\frac{\adj\left(I-M\left(\lambda\right)\right)}{-\tr\left(\adj\left(I-M\left(\lambda\right)\right)\cdot M^{\prime}\left(\lambda\right)\right)}
\]
where $M^{\prime}$ is the entry-wise derivative of $M$, and $\adj A$
is the adjugate or classical adjoint matrix of $A$, that is the transpose
of its cofactor matrix.
\end{thm}

\subsubsection*{Example}

Let $G$ be the directed weighted graph which appears in Figure \ref{fig:Graph-with-two}.

\begin{figure}[h]
\includegraphics[scale=0.7]{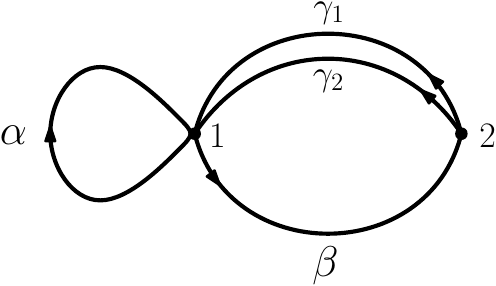}\caption{\label{fig:Graph-with-two}Graph with two vertices $\mathcal{V}=\left\{ 1,2\right\} $
and four edges $\mathcal{E}=\left\{ \alpha,\beta,\gamma_{1},\gamma_{2}\right\} $.}
\end{figure}
\noindent
The graph matrix function of $G$ is given by
\[
M\left(s\right)=\left(\begin{array}{cc}
e^{-l\left(\alpha\right)s} & e^{-l\left(\beta\right)s}\\
e^{-l\left(\gamma_{1}\right)s}+e^{-l\left(\gamma_{2}\right)s} & 0
\end{array}\right).
\]
Putting for example
\begin{eqnarray*}
l\left(\alpha\right)=\log2, &  & l\left(\gamma_{1}\right)=\log\frac{3}{2},\\
l\left(\beta\right)=\log2, &  & l\left(\gamma_{2}\right)=\log3,
\end{eqnarray*}
we get $\lambda=1$ and
\[
Q=\frac{6}{\log432}\left(\begin{array}{cc}
1 & \frac{1}{2}\\
1 & \frac{1}{2}
\end{array}\right),
\]
and so by the second part of Theorem \ref{main result 1}, the number
of paths of length exactly $x$ from vertex $1$ to a point on the
edge $\gamma_{2}$ grows as
\[
\frac{1-e^{-l\left(\gamma_{2}\right)\lambda}}{\lambda}Q_{12}e^{\lambda x}+o\left(e^{\lambda x}\right)=\frac{e^{x}}{\log\sqrt{432}}+o\left(e^{x}\right),\quad x\rightarrow\infty.
\]

\subsection{Weighted random walks on graphs}

Let $\alpha\in\mathcal{E}$ be an edge which originates at $i\in\mathcal{V}$.
Denote by $p_{i\alpha}>0$ the probability that a walker who is passing
through vertex $i$ chooses to continue his walk through edge $\alpha$,
and assume that the sum of the probabilities over all edges originating
at a given vertex is less than or equal to $1$, for all vertices
in $G$. Let $\alpha_{1},\ldots,\alpha_{k}$ be the edges connecting
vertex $i$ to vertex $j$. The graph probability matrix function
$N:\mathbb{C}\rightarrow M_{n}\left(\mathbb{C}\right)$ is defined
by
\[
N_{ij}\left(s\right)=p_{i\alpha_{_{1}}}e^{-s\cdot l\left(\alpha_{_{1}}\right)}+\cdots+p_{i\alpha_{k}}e^{-s\cdot l\left(\alpha_{k}\right)}
\]
and $N_{ij}\left(s\right)=0$ if $i$ is not connected to $j$ by
an edge. Note that the restriction of $N$ to $\mathbb{R}$ is real
valued. If the sum of the probabilities over all edges originating
at a given vertex is strictly less than $1$, there is a positive
probability that the walker does not choose any of the edges and instead
leaves the graph.
\begin{thm}
\label{Main result 2}Let $G$ be a strongly connected incommensurable
graph, and consider a walker on $G$ advancing at constant speed $1$.
There exist a non-positive constant $\lambda$ and a matrix $Q\in M_{n}\left(\mathbb{R}\right)$
with positive entries such that


$\left(ii\right)$\footnote{Previous versions of this paper included a result that appeared as part $\left(i\right)$ of Theorem 2. We have found a mistake in the formulation of this result, and it is therefore omitted from this version.}  Let $\alpha\in\mathcal{E}$ be an edge in $G$
which originates in vertex $j\in\mathcal{V}$. The probability that
a walker who has left some vertex $i\in\mathcal{V}$ at time $t=0$
is on the edge $\alpha\in\mathcal{E}$ at time $t=T$, where $\alpha$
originates at $j$ and has probability $p_{j\alpha}$, decays as
\[
p_{j\alpha}\frac{1-e^{-l\left(\alpha\right)\lambda}}{\lambda}Q_{ij}e^{\lambda T}+o\left(e^{\lambda T}\right),\quad T\rightarrow\infty
\]
whenever $\lambda<0$. In the case $\lambda=0$, the probability tends
to
\[
p_{j\alpha}l\left(\alpha\right)Q_{ij},\quad T\rightarrow\infty.
\]
As in the previous theorem, the constant $\lambda$ is the maximal
real value for which the spectral radius of $N$ is equal to $1$,
and
\[
Q=Q\left(N\left(\lambda\right)\right)=\frac{\adj\left(I-N\left(\lambda\right)\right)}{-\tr\left(\adj\left(I-N\left(\lambda\right)\right)\cdot N^{\prime}\left(\lambda\right)\right)}.
\]
\end{thm}
As a direct corollary we have
\begin{cor}
In the settings of the previous theorem, denote by $\mathcal{E}\left(j\right)$
the set of edges in $G$ with origin at vertex $j\in\mathcal{V}$.
The probability that a walker who has left vertex $i\in\mathcal{V}$
at time zero is still on the graph $G$ at time $T$ decays as
\[
\sum_{j\in\mathcal{V}}\sum_{\alpha\in\mathcal{E}\left(j\right)}p_{j\alpha}\frac{1-e^{-l\left(\alpha\right)\lambda}}{\lambda}Q_{ij}e^{\lambda T}+o\left(e^{\lambda T}\right),\quad T\rightarrow\infty
\]
whenever $\lambda<0$.
\end{cor}
\begin{rem}
It will follow from Remark \ref{gamma is zero remark} that in the
case $\lambda=0$ the matrix $N\left(0\right)$ is right stochastic
and that the probability described in the previous corollary is $1$.

The random walk defined above can be generalized by considering a
random walk on the edges, and the transition probability $p_{\beta,\alpha}$
from an edge $\alpha$ to an edge $\beta$ vanishes unless $\beta$
originates at the vertex where $\alpha$ terminates. Let $d$ be the
number of directed edges on the graph. The graph probability matrix
function $W:\mathbb{C}\rightarrow M_{d}(\mathbb{C})$ is defined by
\[
W_{\beta,\alpha}\left(s\right)=p_{\beta,\alpha}e^{-s\cdot l(\alpha)}.
\]
The analogue of Theorem 2 and its Corollary 1 in the present case
follow directly from the discussion above.
\end{rem}

\subsection{Acknowledgments}

We would like to thank Anton Malyshev, Mikhail Sodin, Yaar Solomon
and Barak Weiss for their valuable contribution to this work. In addition
we thank Thomas Jordan, Zemer Kosloff, Jens Marklof, Frédéric Paulin,
Omri Sarig, Richard Sharp and David Simmons for helpful remarks regarding
relevant results. We are thankful to the referee for several remarks
and suggestions which helped us to improve this work. The first author
is partially supported by the Israel Science Foundation, grant 382/15
and by the United States - Israel Binational Science Foundation, grant
2012037. The second author is supported by the Israel Science Foundation,
grant 2095/15.

\section{Matrices, the Theory of Perron-Frobenius and Graphs}

A real valued matrix $A\in M_{n}\left(\mathbb{R}\right)$ is called\textbf{
positive} if all entries of $A$ are positive and \textbf{non-negative}
if all entries of $A$ are non-negative. $A$ is called \textbf{primitive}
if there exists $k\in\mathbb{N}$ for which $A^{k}$ is positive and
\textbf{irreducible} if for every pair of indices $i,j$ there exists
$k\in\mathbb{N}$ for which $\left(A^{k}\right)_{ij}>0$.

\subsection{The Perron-Frobenius Theorem }

The following results are due to Perron and Frobenius (full statements
and proofs can be found in chapter XIII of \cite{key-3-1}).
\begin{thm*}
Let $A\in M_{n}\left(\mathbb{R}\right)$ be a \textbf{non-negative}
and \textbf{irreducible} matrix.

1. There exists $\mu>0$ which is a simple eigenvalue of $A$, and
$\left|\mu_{j}\right|\leq\mu$ for any other eigenvalue $\mu_{j}$.
We call $\mu$ the \textbf{Perron-Frobenius eigenvalue}.

2. There exist $v,u\in\mathbb{R}^{n}$ with positive entries such
that $Av=\mu v$ and $u^{T}A=\mu u^{T}$. Moreover every right eigenvector
with non-negative entries must be a positive multiple of $v$ \emph{(}similarly
for left eigenvectors and $u$\emph{)}.
\end{thm*}
\begin{thm*}
Let $A\in M_{n}\left(\mathbb{R}\right)$ be a \textbf{primitive} matrix.

1. There exists $\mu>0$ which is a simple eigenvalue of $A$, and
$\left|\mu_{j}\right|<\mu$ for any other eigenvalue $\mu_{j}$. We
call $\mu$ the \textbf{Perron-Frobenius eigenvalue}.

2. There exist $v,u\in\mathbb{R}^{n}$ with positive entries such
that $Av=\mu v$ and $u^{T}A=\mu u^{T}$. Moreover every right eigenvector
with non-negative entries must be a positive multiple of $v$ \emph{(}similarly
for left eigenvectors and $u$\emph{)}.

3. The following limit holds
\[
\lim_{k\rightarrow\infty}\left(\frac{1}{\mu}A\right)^{k}=\frac{vu^{T}}{u^{T}v}.
\]
The limit matrix $P=\frac{vu^{T}}{u^{T}v}$ is called the \textbf{Perron
projection} of $A$.
\end{thm*}

\subsection{Perron's projection}

Given an irreducible matrix $A$, there are additional ways to represent
its Perron projection $P$, as shown bellow
\begin{lem}
\label{Perron =00003D adj and trace}Let $A$ be an irreducible matrix
with Perron-Frobenius eigenvalue $\mu$ and a Perron projection $P$.
Then
\[
P=\frac{\adj\left(\mu I-A\right)}{\tr\left(\adj\left(\mu I-A\right)\right)}.
\]
\end{lem}
\begin{proof}
Let $v$ and $u$ be eigenvectors as in the Perron-Frobenius theorem.
The columns of $P$ are scalar multiples of $v$, and the rows of
$P$ are scalar multiples of $u^{T}$, and so the column space of
$P$ is spanned by $v$ and the row space by $u$. Denote by $V$
the subspace of $M_{n}\left(\mathbb{R}\right)$ consisting of matrices
with these row and column spaces, and notice that $\dim V=1$. Since
$\mu$ is an eigenvalue of $A$ we have
\[
\left(\mu I-A\right)\cdot\adj\left(\mu I-A\right)=\det\left(\mu I-A\right)I=0,
\]
and so every column of $\adj\left(\mu I-A\right)$ is an eigenvector
of $A$ corresponding to $\mu$. By the Perron-Frobenius theorem all
columns of $\adj\left(\mu I-A\right)$ must be scalar multiples of
$v$ and so the column space of $\adj\left(\mu I-A\right)$ is spanned
by $v$. Similarly, using $\adj\left(\mu I-A\right)\cdot\left(\mu I-A\right)=0$
we deduce that the row space of $\adj\left(\mu I-A\right)$ is spanned
by $u$, and so $\adj\left(\mu I-A\right)\in V$. Since $V$ is one
dimensional $\adj\left(\mu I-A\right)=\alpha P$ for some $\alpha\in\mathbb{R}$.
Next, since $Pv=v$, and $Pw=0$ for every $w\in\left(\span\left\{ u\right\} \right)^{\perp}$,
the Perron projection $P$ is similar to the matrix $\diag\left(1,0,\ldots,0\right)$.
Therefore $\tr P=1$, and so
\[
\tr\left(\adj\left(\mu I-A\right)\right)=\tr\left(\alpha P\right)=\alpha\tr P=\alpha,
\]
finishing the proof.
\end{proof}
\begin{cor}
\label{perron projection with derivative of determinant}Let $p_{A}$
be the characteristic polynomial of $A$, then
\[
P=\frac{\adj\left(\mu I-A\right)}{\frac{d}{dx}p_{A}\left(x\right)\vert_{x=\mu}}.
\]
\end{cor}
\begin{proof}
Jacobi's formula for the derivative of the determinant of a matrix
is given by
\[
\frac{d}{dx}\det B\left(x\right)=\tr\left(\adj\left(B\left(x\right)\right)B'\left(x\right)\right)
\]
and so using this formula, the corollary follows from the previous
lemma for $B\left(x\right)=xI-A$.
\end{proof}
\begin{rem}
This result and others concerning the theory of Perron-Frobenius may
be found in \cite{key-12362536246246}. Another proof for Corollary
\ref{perron projection with derivative of determinant} can be derived
by direct computation using the following identity
\[
\adj\left(\lambda I-A\right)=A^{n-1}+\left(\lambda+p_{n-1}\right)A^{n-2}+\cdots+\left(\lambda^{n-1}+p_{n-1}\lambda^{n-2}+\cdots+p_{1}\right)I
\]
where $p_{A}\left(x\right)=x^{n}+p_{n-1}x^{n-1}+\cdots+p_{0}$ and
$\lambda\in\mathbb{R}$ (see \cite{key-3-1} or \cite{key-1-2}).
Let $\mu$ be the Perron-Frobenius eigenvalue and $v$ a corresponding
eigenvector, we compute
\[
\begin{array}{rl}
\adj\left(\mu I-A\right)v & =\left[\mu^{n-1}+\left(\mu+p_{n-1}\right)\mu^{n-2}+\cdots+\left(\mu^{n-1}+p_{n-1}\mu^{n-2}+\cdots+p_{1}\right)\right]v\\
 & =\left[n\mu^{n-1}+\left(n-1\right)p_{n-2}\mu^{n-3}+\cdots+p_{1}\right]v\\
 & =p_{A}'\left(\mu\right)v.
\end{array}
\]
Recall that $\adj\left(\mu I-A\right)=\alpha P$ for some $\alpha\in\mathbb{R}$.
Since
\[
p_{A}'\left(\mu\right)v=\adj\left(\mu I-A\right)v=\alpha Pv=\alpha v,
\]
it follows that $\alpha=\tr\left(\adj\left(\mu I-A\right)\right)=p_{A}'\left(\mu\right)$.
\end{rem}
\begin{cor}
Let $\left(\mu,\mu_{2},\ldots,\mu_{n}\right)$ be the eigenvalues
of $A$, perhaps with repetitions. Then
\[
P=\frac{\prod\left(A-\mu_{i}I\right)}{\prod\left(\mu-\mu_{i}\right)}.
\]
\end{cor}
\begin{proof}
Using the representation of $\adj\left(\mu I-A\right)$ as a polynomial
of degree $n-1$, it follows from Vieta's polynomial formulas (see
for example \cite{key-123}) that
\[
\adj\left(\mu I-A\right)=\left(A-\mu_{2}I\right)\cdots\left(A-\mu_{n}I\right).
\]
Since $p_{A}'\left(\mu\right)=\left(\mu-\mu_{2}\right)\cdots\left(\mu-\mu_{n}\right)$
this gives the desired result.
\end{proof}

\subsection{Comparison between the non-weighted case and the weighted case}

When considering path counting questions on a non-weighted graph $G$,
it is convenient to define its adjacency matrix. This is the square
matrix $A\in M_{n}\left(\mathbb{R}\right)$ indexed by the vertices
of $G$, where $A_{ij}$ is set as the number of edges from vertex
$i$ to vertex $j$. Note that an adjacency matrix of a strongly connected
graph is irreducible, but not necessarily primitive. For primitivity
of the adjacency matrix we must also assume that $G$ is aperiodic,
which means that the greatest common divisor of the set of lengths
of all closed paths is $1$.

Recall that the number of paths from vertex $i$ to vertex $j$ consisting
of exactly $k$ edges is exactly $\left(A^{k}\right)_{ij}$. It follows
that if the graph is strongly connected and aperiodic, then $A$ is
primitive, and by the Perron-Frobenius theorem this number can be
approximated by $P_{ij}\mu^{k}$, where $P$ is the Perron projection
of $A$ described above and $\mu$ is the Perron-Frobenius eigenvalue.

It is interesting to compare the matrices $P$ and $Q$, where $Q$
is the matrix defined in the statement of Theorem \ref{main result 1}.
Due to Jacobi's formula we can write
\[
Q=\frac{\adj\left(I-M\left(\lambda\right)\right)}{-\tr\left(\adj\left(I-M\left(\lambda\right)\right)\cdot M^{\prime}\left(\lambda\right)\right)}=\frac{\adj\left(I-M\left(\lambda\right)\right)}{\frac{d}{ds}\left(\det\left(I-M\left(s\right)\right)\right)\vert_{s=\lambda}}
\]
and
\[
P=\frac{\adj\left(I-\frac{1}{\mu}A\right)}{\tr\left(\adj\left(I-\frac{1}{\mu}A\right)\right)}=\frac{\adj\left(I-\frac{1}{\mu}A\right)}{\frac{d}{dx}\det\left(I-\frac{1}{x}\frac{A}{\mu}\right)\vert_{x=1}}
\]
and the resemblance is clear.

We remark that in the case of a non-weighted graph $G$ we assume
that $G$ is strongly connected and aperiodic in order to guarantee
convergence of $\frac{1}{\mu^{k}}A^{k}$ to $P$, otherwise the corresponding
adjacency matrix need not be primitive and the Perron-Frobenius theorem
may not be implied. In the case of weighted graphs we replace the
assumption that $G$ is aperiodic by the assumption of incommensurability.

As an example we look at the following case: Assume all edges in $G$
are of equal length $a>0$. So $M\left(s\right)=e^{-as}A$ where $A$
is the adjacency matrix of the underlying non-weighted graph. Obviously
$G$ is not incommensurable and the assumptions of Theorem \ref{main result 1}
do not hold, but still we can calculate $Q$.

Let $\mu$ be the Perron-Frobenius eigenvalue of $A$, then the matrix
$M\left(\frac{\log\mu}{a}\right)=\frac{1}{\mu}A$ has Perron-Frobenius
eigenvalue $1$, and so $\lambda=\frac{\log\mu}{a}$. Since
\[
\begin{array}{rl}
-\tr\left(\adj\left(I-M\left(\lambda\right)\right)\cdot M^{\prime}\left(\lambda\right)\right) & =-\tr\left(\adj\left(I-\frac{1}{\mu}A\right)\cdot\frac{-a}{\mu}A\right)\\
 & =a\tr\left(\adj\left(I-\frac{1}{\mu}A\right)\cdot\frac{1}{\mu}A\right)\\
 & =a\tr\left(\adj\left(I-\frac{1}{\mu}A\right)\right)
\end{array}
\]
we get
\[
Q=\frac{\adj\left(I-M\left(\lambda\right)\right)}{-\tr\left(\adj\left(I-M\right)\cdot M^{\prime}\left(\lambda\right)\right)}=\frac{\adj\left(I-\frac{1}{\mu}A\right)}{a\tr\left(\adj\left(I-\frac{1}{\mu}A\right)\right)}=\frac{1}{a}P
\]
and so if we think of a non-weighted graph as a weighted graph with
edges all of length $a=1$, we get $P=Q$.

\section{The Wiener-Ikehara Theorem and the Laplace Transform}

\subsection{The Wiener-Ikehara Theorem}

The proofs of our main results follow from this Tauberian theorem
due to Wiener and Ikehara (see chapter 8.3 in \cite{key-5}).
\begin{thm*}
Let $f(x)$ be a non-negative and monotone function on $[0,\infty)$.
Suppose that the Laplace transform of $f\left(x\right)$, given by
\[
F(s):=\mathcal{L}\left\{ f\left(x\right)\right\} \left(s\right)=\int_{0}^{\infty}f\left(x\right)e^{-xs}dx,
\]
converges for all $s$ with $\mbox{Re}(s)>\lambda$, and that there
exists $c\in\mathbb{R}$ for which the function
\[
F(s)-\frac{c}{s-\lambda}
\]
extends to a continuous function in the closed half-plane $\mbox{Re}(s)\geq\lambda$.
Then
\[
f\left(x\right)=ce^{\lambda x}+o\left(e^{\lambda x}\right),\quad x\rightarrow\infty.
\]
\end{thm*}

\subsection{\label{subsec:The-Laplace-Transform}The Laplace Transform of the
counting and probability functions}

Denote by $\Gamma\left(i,j\right)$ the set of paths originating at
vertex $i\in\mathcal{V}$ and terminating at vertex $j\in\mathcal{V}$,
and by $p\left(\gamma\right)$ the product of probabilities of the
edges which define the path $\gamma$.

Let $A_{i,j}\left(x\right)$ denote the number of paths originating
at vertex $i$ and terminating at vertex $j$ of length at most $x$.
Then
\[
A_{i,j}\left(x\right)=\sum_{\gamma\in\Gamma\left(i,j\right)}\chi_{[l\left(\gamma\right),\infty)}\left(x\right)=\sum_{k=0}^{\infty}\sum_{\begin{subarray}{c}
\gamma\in\Gamma\left(i,j\right)\\
\text{with\,\ensuremath{k}\,edges}
\end{subarray}}\chi_{[l\left(\gamma\right),\infty)}\left(x\right)
\]
where $\chi_{A}$ is the characteristic function of the set $A\subset\mathbb{R}$.
The Laplace transform is
\[
\mathcal{L}\left\{ A_{i,j}\left(x\right)\right\} \left(s\right)=\sum_{k=0}^{\infty}\sum_{\begin{subarray}{c}
\gamma\in\Gamma\left(i,j\right)\\
\text{with\,\ensuremath{k}\,edges}
\end{subarray}}\frac{1}{s}e^{-l\left(\gamma\right)s}=\frac{1}{s}\left(\sum_{k=0}^{\infty}M^{k}\left(s\right)\right)_{i,j}.
\]

Let $\alpha$ be an edge originating at vertex $j$. Denote by $B_{i,\alpha}\left(x\right)$
the number of paths of length exactly $x$ from vertex $i$ to a point
on the edge $\alpha$. Then
\[
B_{i,\alpha}\left(x\right)=\sum_{\gamma\in\Gamma\left(i,j\right)}\chi_{[l\left(\gamma\right),l\left(\gamma\right)+l\left(\alpha\right))}\left(x\right)=\sum_{k=0}^{\infty}\sum_{\begin{subarray}{c}
\gamma\in\Gamma\left(i,j\right)\\
with\,k\,edges
\end{subarray}}\chi_{[l\left(\gamma\right),l\left(\gamma\right)+l\left(\alpha\right))}\left(x\right).
\]
The Laplace transform is
\[
\mathcal{L}\left\{ B_{i,\alpha}\left(x\right)\right\} \left(s\right)=\sum_{k=0}^{\infty}\sum_{\begin{subarray}{c}
\gamma\in\Gamma\left(i,j\right)\\
\text{with\,\ensuremath{k}\,edges}
\end{subarray}}\frac{1-e^{-l\left(\alpha\right)s}}{s}e^{-l\left(\gamma\right)s}=\frac{1-e^{-l\left(\alpha\right)s}}{s}\left(\sum_{k=0}^{\infty}M^{k}\left(s\right)\right)_{i,j}.
\]


Denote by $D_{i,\alpha}\left(T\right)$ the probability that a walker
leaving $i$ at time zero and moving along the graph at speed $1$,
would at time $T$ be on the edge $\alpha$ which originates at vertex
$j$. Then
\[
D_{i,\alpha}\left(T\right)=\sum_{\gamma\in\Gamma\left(i,j\right)}p\left(\gamma\right)p_{ij}\chi_{[l\left(\gamma\right),l\left(\gamma\right)+l\left(\alpha\right))}\left(T\right)=p_{ij}\sum_{k=0}^{\infty}\sum_{\begin{subarray}{c}
\gamma\in\Gamma\left(i,j\right)\\
\text{with\,\ensuremath{k}\,edges}
\end{subarray}}p\left(\gamma\right)\chi_{[l\left(\gamma\right),l\left(\gamma\right)+l\left(\alpha\right))}\left(T\right).
\]
The Laplace transform is
\[
\mathcal{L}\left\{ D_{i,\alpha}\left(T\right)\right\} \left(s\right)=\sum_{k=0}^{\infty}\sum_{\begin{subarray}{c}
\gamma\in\Gamma\left(i,j\right)\\
\text{with\,\ensuremath{k}\,edges}
\end{subarray}}p_{j\alpha}\frac{1-e^{-l\left(\alpha\right)s}}{s}p\left(\gamma\right)e^{-l\left(\gamma\right)s}=p_{j\alpha}\frac{1-e^{-l\left(\alpha\right)s}}{s}\left(\sum_{k=0}^{\infty}N^{k}\left(s\right)\right)_{i,j}.
\]

It will be shown that the sums $\sum_{k=0}^{\infty}M^{k}\left(s\right)$
and $\sum_{k=0}^{\infty}N^{k}\left(s\right)$ converge absolutely
for suitable values of $s$, and so we can change the order of summation
and integration as implied in the calculations above.

We will show that the constant $\lambda$ as described in the statement
of the Theorems \ref{main result 1} and \ref{Main result 2} exists,
and that these Laplace transforms satisfy the conditions of the Wiener-Ikehara
theorem with a simple pole at $s=\lambda$.

\section{Proof of main results }

Although some of the following results can be found in the literature,
we include the full details for the sake of clarity.
\begin{lem}
\label{lem: inequality complex matrix dominated by real}The matrix
elements of powers of $M\left(s\right)$ for $s=\sigma+\mathbf{i}t$
(respectively $N\left(s\right)$) are bounded in absolute value by
the corresponding matrix elements of powers of $M\left(\sigma\right)$
(respectively $N\left(\sigma\right)$).
\end{lem}
\begin{proof}
Indeed for every $k\in\mathbb{N}$
\begin{eqnarray*}
\left|\left(M^{k}\left(s\right)\right)_{ij}\right| & = & \left|\sum_{i_{1},..,i_{k-1}}M_{i,i_{1}}\left(s\right)\cdots M_{i_{k-1},j}\left(s\right)\right|\\
 & \leq & \sum_{i_{1},..,i_{k-1}}\left|M_{i,i_{1}}\left(s\right)\right|\cdots\left|M_{i_{k-1},j}\left(s\right)\right|\\
 & \leq & \sum_{i_{1},..,i_{k-1}}M_{i,i_{1}}\left(\sigma\right)\cdots M_{i_{k-1},j}\left(\sigma\right)=\left(M^{k}\left(\sigma\right)\right)_{ij}
\end{eqnarray*}
and similarly for $N$, as required.
\end{proof}
\begin{rem}
This lemma is contained in a result due to Wielandt which can be found
in \cite{key-3-1}.
\end{rem}
For $\sigma\in\mathbb{R}$ the matrices $M\left(\sigma\right)$ and
$N\left(\sigma\right)$ are real, non-negative and irreducible (because
the graph $G$ is strongly connected), and so by Perron-Frobenius
there exists a dominant real eigenvalue $\mu\left(\sigma\right)$
of multiplicity $1$ corresponding to a positive eigenvector $v\left(\sigma\right)$.
\begin{lem}
\label{lem:lambda for with PF eigenval is 1}Let $M\left(\sigma\right)$
be as above. Then there exists $\lambda>0$ such that $\mu\left(\lambda\right)=1$
and for every $\sigma>\lambda$ the corresponding dominant eigenvalue
satisfies $\mu\left(\sigma\right)<1$.
\end{lem}
\begin{proof}
For all $\sigma\in\mathbb{R}$ there exists $\mu\left(\sigma\right)$
which by Perron-Frobenius is a simple eigenvalue of $M\left(\sigma\right)$.
Let $v\left(\sigma\right)$ and $u\left(\sigma\right)$ be right and
left positive eigenvectors, then since $M\left(\sigma\right)$ is
differentiable, then by Theorem 6.3.12 in \cite{key-3-2} $\mu\left(\sigma\right)$
is differentiable and the following formula holds
\[
\frac{d}{d\sigma}\mu\left(\sigma\right)=\frac{u^{T}\left(\sigma\right)M^{\prime}\left(\sigma\right)v\left(\sigma\right)}{u^{T}\left(\sigma\right)v\left(\sigma\right)}.
\]
Since the eigenvectors are positive, and the entry-wise derivative
of $M$ is non-positive, we deduce that
\begin{equation}
\frac{d}{d\sigma}\mu\left(\sigma\right)<0\label{eq:derivative of PF eigenval is negative}
\end{equation}
and in particular $\mu$ is monotone decreasing. Recall that $\mu\left(0\right)$
is the largest eigenvalue of the adjacency matrix $M\left(0\right)$
of the strongly connected and incommensurable graph $G$ and so\textbf{
$\mu\left(0\right)>1$.} Moreover, since all elements of $M\left(\sigma\right)$
tend to zero as $\sigma$ tends to infinity, so does the Perron-Frobenius
eigenvalue. Therefore there exists a finite $\lambda>0$ for which
$\mu\left(\lambda\right)=1$ and $\mu\left(\sigma\right)<1$ for all
$\sigma>\lambda$.
\end{proof}
\begin{lem}
\label{probability dominant eigenvalue}Let $N\left(\sigma\right)$
be as above. Then there exists $\lambda\leq0$ such that $\mu\left(\lambda\right)=1$
and for every $\sigma>\lambda$ the corresponding dominant eigenvalue
satisfies $\mu\left(\sigma\right)<1$.
\end{lem}
\begin{proof}
The proof is similar to the discussion about $M\left(\sigma\right)$,
only here we must show that $\mu\left(0\right)\leq1$ to verify that
the value of $\lambda$ for which $\mu\left(\lambda\right)=1$ is
negative. This follows from our assumption that the sum of the probabilities
of edges originating at a given vertex is less or equal to $1$, and
so the sum of the entries of any row in $N\left(0\right)$ is bounded
by $1$, therefore $\mu\left(0\right)$ is bounded by $1$ (see Wielandt's\textbf{
}proof of Perron-Frobenius theorem which appears in \cite{key-3-1}).
\end{proof}
\begin{rem}
\label{gamma is zero remark}Clearly $\mu\left(0\right)=1$ if and
only if the sum of all probabilities for edges originating at vertex
$i$ is~$1$, for all $i$. In other words $\lambda=0$ if and only
if $N\left(0\right)$ is a right stochastic matrix, that is all its
rows sum up to $1$.
\end{rem}
The following Lemmas are stated for $M$, but analogous statements
and their proofs apply for $N$.
\begin{lem}
\label{generating function convergence and formula}Let $\lambda\in\mathbb{R}$
be as in Lemma \ref{lem:lambda for with PF eigenval is 1}. Then $\sum_{k=0}^{\infty}M^{k}\left(\sigma\right)$
converges for all $s=\sigma+\mathbf{i}t$ with $\sigma>\lambda$,
and in this case
\[
\sum_{k=0}^{\infty}M^{k}\left(s\right)=\left(I-M\left(s\right)\right)^{-1}=\frac{\adj\left(I-M\left(s\right)\right)}{\det\left(I-M\left(s\right)\right)}
\]
and so the Laplace transforms of the counting functions described
above are analytic in the half plane $\sigma>\lambda$.
\end{lem}
\begin{proof}
The Lemma follows because for $\sigma>\lambda$, as in Lemma \ref{lem:lambda for with PF eigenval is 1},
the geometric sum $\sum_{k=0}^{\infty}M^{k}\left(\sigma\right)$ converges,
and by Lemma \ref{lem: inequality complex matrix dominated by real}
so does $\sum_{k=0}^{\infty}M^{k}\left(s\right)$ .
\end{proof}
Plugging the previous statement in the expressions derived in the
previous section for the Laplace transforms of the counting functions
we study, we conclude the following corollary.
\begin{cor}
\label{Laplace transforms}Let $A_{i,j}\left(x\right),\,B_{i,\alpha}\left(x\right)$
and $D_{i,\alpha}\left(T\right)$ be the counting functions defined
in subsection \ref{subsec:The-Laplace-Transform}. The associated
Laplace transforms are given by
\begin{eqnarray*}
\mathcal{L}\left\{ A_{i,j}\left(x\right)\right\} \left(s\right) & = & \frac{1}{s}\cdot\frac{\left({\rm adj}\left(I-M\left(s\right)\right)\right)_{ij}}{\det\left(I-M\left(s\right)\right)},\\
\mathcal{L}\left\{ B_{i,\alpha}\left(x\right)\right\} \left(s\right) & = & \frac{1-e^{-l\left(\alpha\right)s}}{s}\cdot\frac{\left({\rm adj}\left(I-M\left(s\right)\right)\right)_{ij}}{\det\left(I-M\left(s\right)\right)},\\
\mathcal{L}\left\{ D_{i,\alpha}\left(T\right)\right\} \left(s\right) & = & p_{j\alpha}\frac{1-e^{-l\left(\alpha\right)s}}{s}\cdot\frac{\left({\rm adj}\left(I-N\left(s\right)\right)\right)_{ij}}{\det\left(I-N\left(s\right)\right)},
\end{eqnarray*}
where $\frac{1-e^{-l\left(\alpha\right)s}}{s}$ is an entire function
with value $l\left(\alpha\right)$ at $s=0$.
\end{cor}
\begin{lem}
The matrix $\mbox{\ensuremath{\adj\left(I-M\left(\lambda\right)\right)}}$
has positive entries.
\end{lem}
\begin{proof}
Recall that $\mu\left(\lambda\right)=1$, and so by Lemma \ref{Perron =00003D adj and trace}
there exist positive vectors $v,u$ such that
\[
\frac{\adj\left(I-M\left(\lambda\right)\right)}{\tr\left(\adj\left(I-M\left(\lambda\right)\right)\right)}=\frac{vu^{T}}{u^{T}v}.
\]
It follows that all the entries of\textbf{ $\mbox{\ensuremath{\adj\left(I-M\left(\lambda\right)\right)}}$}
are non-zero and have the same sign as $\tr\left(\adj\left(I-M\left(\lambda\right)\right)\right)$,
which is $\frac{d}{dx}p_{M_{G}\left(\lambda_{G}\right)}\left(x\right)\vert_{x=1}$
by Jacobi's formula. But $1$ is a simple root of the characteristic
polynomial and is the largest one, and therefore its derivative at
$x=1$ is positive.
\end{proof}
\begin{lem}
\label{simple pole}The Laplace transforms of the graph counting functions
$A_{i,j}\left(x\right)$ and $B_{i,\alpha}\left(x\right)$ have a
simple pole at $\lambda$.
\end{lem}
\begin{proof}
The point $s=\lambda$ is a singular point of the Laplace transforms,
because by the previous lemma the numerator $\left(\mbox{adj}\left(I-M\left(\lambda\right)\right)\right)_{ij}$
is non-zero while the denominator has a zero at $\lambda$. It is
thus enough to show that the zero of $\det\left(I-M\left(s\right)\right)$
at $\lambda$ is a simple one. For $\sigma\in\mathbb{R}$, the characteristic
polynomial of $M\left(\sigma\right)$ is given by
\[
p_{M\left(\sigma\right)}\left(x\right)=\det\left(xI-M\left(\sigma\right)\right)=\left(x-\mu\left(\sigma\right)\right)\left(x-\mu_{2}\left(\sigma\right)\right)\cdots\left(x-\mu_{n}\left(\sigma\right)\right),
\]
where $\mu(\sigma)$ is the Perron-Frobenius eigenvalue of $M\left(\sigma\right)$.
Therefore $\mu\left(\lambda\right)=1$, and $\mu_{j}\left(\sigma\right)\not=1$
for $j\geq2$ in a small neighborhood of $\lambda$, and
\[
\det\left(I-M\left(\sigma\right)\right)=\left(1-\mu\left(\sigma\right)\right)\cdots\left(1-\mu_{n}\left(\sigma\right)\right).
\]
It follows from equation \eqref{eq:derivative of PF eigenval is negative}
that the function $\left(1-\mu\left(\sigma\right)\right)$ has a simple
zero at $\lambda$, and the same holds for the function $\det\left(I-M\left(\sigma\right)\right)$
and therefore also for $\det\left(I-M\left(s\right)\right)$.
\end{proof}
\begin{lem}
\label{no other poles on critical line}For all $t\not=0$,
\[
\det\left(I-M\left(\lambda+\mathbf{i}t\right)\right)\not=0,
\]
that is the Laplace transforms have no other poles on the line $\mbox{Re}\left(s\right)=\lambda$
than at $s=\lambda$ itself.
\end{lem}
\begin{proof}
Say that $G$ has the \textbf{single-edge property} if for every pair
of vertices $i,j$ in $G$ there is at most one edge from vertex $i$
to vertex $j$. Given a graph $G$, we define a graph $\widetilde{G}$
by adding a new vertex in the middle of every edge in $G$. There
is a natural one to one map between paths in $G$ and paths in $\widetilde{G}$
which originate and terminate in the original set of vertices, and
clearly $\widetilde{G}$ has the single-edge property. Therefore,
there is no loss of generality assuming that this property holds for
$G$.

Put $M_{ij}=\left(M\left(\lambda\right)\right)_{ij}$. Since $G$
has the single-edge property, the entries of the matrix $M\left(\lambda\right)$
are either $M_{ij}=e^{-\lambda\cdot l\left(\alpha\right)}$ if there
is an edge $\alpha\in\mathcal{E}$ connecting the vertex $i$ to the
vertex $j$, or $M_{ij}=0$ if there is no such edge. Let $v=\left(v_{1},...,v_{n}\right)$
be a positive eigenvector of $M\left(\lambda\right)$ corresponding
to the eigenvalue $\mu\left(\lambda\right)=1$ and define $D$ to
be the following invertible matrix
\[
D=\diag\left(v_{1},...,v_{n}\right),\,\,\,\,D^{-1}=\diag\left(\frac{1}{v_{1}},...,\frac{1}{v_{n}}\right).
\]
The matrix given by $S=D^{-1}MD$ is a non-negative and right stochastic
matrix. Indeed, since $S_{ij}=M_{ij}v_{j}/v_{i}$, it is clear that
$S_{ij}\geq0$ and that the sum of the elements of the $i$th row
is
\[
\sum_{j=1}^{n}S_{ij}=\sum_{j=1}^{n}M_{ij}\frac{v_{j}}{v_{i}}=\frac{1}{v_{i}}\sum_{j=1}^{n}M_{ij}v_{j}=\frac{v_{i}}{v_{i}}=1.
\]

Let $S\left(s\right)$ be the matrix with coefficients from $S$ raised
to the power of $s$, that is
\[
\left(S\left(s\right)\right)_{ij}=\left(S_{ij}\right)^{s}=\left(M_{ij}\right)^{s}\left(\frac{v_{j}}{v_{i}}\right)^{s}.
\]
Notice that
\[
S\left(s\right)=\left(D^{s}\right)^{-1}M\left(\lambda s\right)D^{s},
\]
that is $M\left(\lambda s\right)$ and $S\left(s\right)$ are similar,
and in particular they have the same characteristic polynomial $p_{S\left(s\right)}\left(x\right)=p_{M\left(\lambda s\right)}\left(x\right)$.
Recalling the definition of the characteristic polynomial and plugging
$x=1$ we see that
\[
\det\left(I-M\left(\lambda s\right)\right)=\det\left(I-S\left(s\right)\right)
\]
and so it is enough to show that $\det\left(I-S\left(1+\mathbf{i}t\right)\right)\not=0$
for all $t\not=0$.

The following argument is due to Parry (see \cite{key-6}). Assume
$\det\left(I-S\left(1+\mathbf{i}t\right)\right)=0$ for some $t\not=0$.
So there exists a non-zero vector $u=\left(u_{1},...,u_{n}\right)$
for which
\[
S\left(1+\mathbf{i}t\right)u=u,
\]
that is for all $i$
\[
\sum_{j=1}^{n}S_{ij}^{1+\mathbf{i}t}u_{j}=\sum_{j=1}^{n}S_{ij}S_{ij}^{\mathbf{i}t}u_{j}=u_{i}.
\]
By the triangle inequality, for all $i$
\[
\left|u_{i}\right|\leq\sum_{j=1}^{n}\left|S_{ij}^{1+\mathbf{i}t}u_{j}\right|=\sum_{j=1}^{n}S_{ij}\left|S_{ij}^{\mathbf{i}t}\right|\left|u_{j}\right|=\sum_{j=1}^{n}S_{ij}\left|u_{j}\right|,
\]
and together with the equality
\[
\sum_{j=1}^{n}S_{ij}=1
\]
this implies that
\[
\left|u_{1}\right|=...=\left|u_{n}\right|.
\]
Assume $\left|u_{j}\right|=r>0$ for all $j$, and notice that this
means $S_{ij}^{\mathbf{i}t}u_{j}$ are points on a circle of radius
$r$. We have therefore that every $u_{i}$, which is itself a point
on the circle of radius $r$, is a convex combination (that is, a
linear combination with positive coefficients all adding up to $1$)
of points on that same circle. This is only possible of course if
$S_{ij}^{\mathbf{i}t}u_{j}=u_{i}$ for all $j$ such that $S_{ij}\not=0$.

Now, for any closed orbit on the graph, let $\alpha_{1}=\left(k_{1},k_{2}\right),...,\alpha_{m}=\left(k_{m},k_{1}\right)$
denote the corresponding sequence of edges. We get
\[
\left(S_{k_{1}k_{2}}^{\mathbf{i}t}u_{k_{2}}\right)\left(S_{k_{2}k_{3}}^{\mathbf{i}t}u_{k_{3}}\right)...\left(S_{k_{m-1}k_{m}}^{\mathbf{i}t}u_{k_{m}}\right)\left(S_{k_{m}k_{1}}^{\mathbf{i}t}u_{k_{1}}\right)=u_{k_{1}}u_{k_{2}}...u_{k_{m-1}}u_{k_{m}}
\]
and so
\[
S_{k_{1}k_{2}}^{\mathbf{i}t}...S_{k_{m}k_{1}}^{\mathbf{i}t}=\left(S_{k_{1}k_{2}}...S_{k_{m}k_{1}}\right)^{\mathbf{i}t}=1,
\]
which gives
\[
\left(M_{k_{1}k_{2}}...M_{k_{m}k_{1}}\right)^{\mathbf{i}t}=1.
\]
But
\[
M_{k_{1}k_{2}}...M_{k_{m}k_{1}}=e^{-\left(l(\alpha_{1})+...+l(\alpha_{m})\right)}
\]
and so there exists some $l\in\mathbb{Z}$ for which
\[
t=\frac{2\pi l}{l\left(\alpha_{1}\right)+...+l\left(\alpha_{m}\right)}.
\]
This hold for every closed orbit on the graph, which is a contradiction
to our irrationality assumptions on the lengths of the closed orbits
on our incommensurable graph $G$.
\end{proof}
\begin{rem}
When proving the analogous Lemma for the case of the matrix $N$,
simply assign the probability $1$ to all edges of $\widetilde{G}$
originating at vertices of $\widetilde{G}$ which are not in $G$.
\end{rem}
\begin{rem}
There is another construction of a graph $\widehat{G}$ associated
to $G$ which preserves its structure and has the single-edge property.
Let $i\in\mathcal{V}$ be a vertex in $G$ and assume that there are
$k_{i}$ distinct edges which terminate at $i$ and $l_{i}$ distinct
edges which originate at $i$ as shown in Figure \ref{fig:Covering graph}.

\begin{figure}[H]
\includegraphics[scale=0.8]{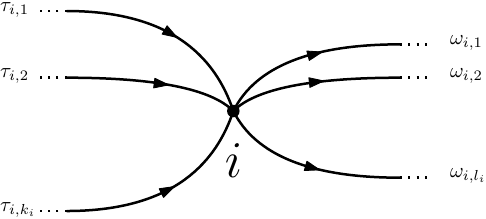}

\caption{\label{fig:Covering graph}Vertex $i\in\mathcal{V}$ and the edges
which terminate and originate at $i$ in $G$}
\end{figure}

Define $k_{i}\cdot l_{i}$ associated vertices in $\widehat{G}$,
indexed by
\[
\begin{array}{ccc}
\tau_{i1}-i-\omega_{i1}, & \ldots & \tau_{ik_{i}}-i-\omega_{i1}\\
\vdots &  & \vdots\\
\tau_{i1}-i-\omega_{il_{i}}, & \ldots & \tau_{ik_{i}}-i-\omega_{il_{i}}
\end{array}
\]
and repeat this procedure for all vertices in $G$. Define $\widehat{G}$
to contain a directed edge from $\widehat{v}_{1}=\tau_{is}-i-\omega_{it}$
to $\widehat{v}_{2}=\tau_{ju}-j-\omega_{jv}$ if and only if $\omega_{it}=\tau_{ju}$.
If there is such an edge, and if $\alpha$ is the edge in $G$ for
which $\alpha=\omega_{it}=\tau_{ju}$, then the edge between $\widehat{v}_{1}$
and $\widehat{v}_{2}$ is labeled by $\alpha$.

It can be shown that for every path $p$ in $G$ which originates
in vertex $i$ and terminates in vertex $j$, there are exactly $k_{i}\cdot l_{j}$
distinct paths in $\widehat{G}$ which are copies of $p$ in the sense
of the edges they consist of. It follows that functions counting paths
in $G$ which concern vertices $i$ and $j$ differ by a multiplicative
constant from the associated functions on $\widehat{G}$, and the
same holds for their Laplace transforms. As a result, the position
of the poles of the Laplace transforms is not changed.
\end{rem}
\begin{lem}
\label{residue at lambda}The residue of the function $\frac{\adj\left(I-M\left(s\right)\right)_{ij}}{\det\left(I-M\left(s\right)\right)}$
at $s=\lambda$ is
\[
Q_{ij}=\frac{\left(\adj\left(I-M\left(\lambda\right)\right)\right)_{ij}}{-\tr\left(\adj\left(I-M\left(\lambda\right)\right)\cdot M^{\prime}\left(\lambda\right)\right)}
\]
\end{lem}
\begin{proof}
We have seen that the function has a simple pole at $s=\lambda$,
and so the residue at $s=\lambda$ is
\[
\frac{\left(\adj\left(I-M\left(\lambda\right)\right)\right)_{ij}}{\frac{d}{ds}\left(\det\left(I-M\left(s\right)\right)\right)\vert_{s=\lambda}}.
\]
Finally, use Jacobi's formula to obtain
\begin{eqnarray*}
\frac{d}{ds}\left(\det\left(I-M\left(s\right)\right)\right) & = & \tr\left(\adj\left(I-M\left(\lambda\right)\right)\cdot\left(I-M\left(s\right)\right)^{\prime}\left(\lambda\right)\right)\\
 & = & -\tr\left(\adj\left(I-M\left(\lambda\right)\right)\cdot M^{\prime}\left(\lambda\right)\right).
\end{eqnarray*}
Combining the above we get the desired formula for the residue at
hand.
\end{proof}

\subsection{Proof of Theorems \ref{main result 1} and \ref{Main result 2}}

We show how to conclude the main statements of this paper from the
Lemmas proven above.
\begin{proof}[Proof of Theorem \ref{main result 1}]
Let $G$ be a strongly connected incommensurable graph and let $M$
be the graph matrix function of $G$. For $\sigma\in\mathbb{R}$ the
matrix $M\left(\sigma\right)$ is real and due to Perron-Frobenius
there exists a dominant real eigenvalue $\mu\left(\sigma\right)$
of multiplicity $1$. By Lemma \ref{lem:lambda for with PF eigenval is 1}
there exists $\lambda>0$ such that $\mu\left(\lambda\right)=1$ and
for every $\sigma>\lambda$ the corresponding dominant eigenvalue
satisfies $\mu\left(\sigma\right)<1$. By Lemma \ref{lem: inequality complex matrix dominated by real}
the series $\sum_{k=0}^{\infty}M^{k}\left(s\right)$ converges for
all $s=\sigma+\mathbf{i}t$ with $\sigma>\lambda$ to $\frac{\adj\left(I-M\left(s\right)\right)}{\det\left(I-M\left(s\right)\right)}$
and so by Corollary \ref{Laplace transforms}
\[
\mathcal{L}\left\{ A_{i,j}\left(x\right)\right\} \left(s\right)=\frac{1}{s}\cdot\frac{\left(\mbox{adj}\left(I-M\left(s\right)\right)\right)_{ij}}{\det\left(I-M\left(s\right)\right)}.
\]
By Lemma \ref{simple pole} this $\mathcal{L}\left\{ A_{i,j}\left(x\right)\right\} \left(s\right)$
has a simple pole at $s=\lambda$ and by Lemma \ref{no other poles on critical line}
there are no other poles on the line $\mbox{Re}\left(s\right)=\lambda$.
By Lemma \ref{residue at lambda} the residue of $\mathcal{L}\left\{ A_{i,j}\left(x\right)\right\} \left(s\right)$
at $s=\lambda$ is
\[
\frac{1}{\lambda}\frac{\left(\adj\left(I-M\left(\lambda\right)\right)\right)_{ij}}{-\tr\left(\adj\left(I-M\left(\lambda\right)\right)\cdot M^{\prime}\left(\lambda\right)\right)}=\frac{1}{\lambda}Q_{ij}.
\]
Therefore applying the Wiener-Ikehara Theorem yields statement $\left(i\right)$,
namely that the number of paths from $i\in\mathcal{V}$ to $j\in\mathcal{V}$
of length at most $x$ grows as
\[
\frac{1}{\lambda}Q_{ij}e^{\lambda x}+o\left(e^{\lambda x}\right),\quad x\rightarrow\infty.
\]

Replacing $A_{i,j}\left(x\right)$with $B_{i,j}\left(x\right)$ and
repeating these steps yields statement $\left(ii\right)$, namely
that the number of paths of length exactly $x$ from some vertex $i$
to a point on the edge $\alpha$ grows a
\[
\frac{1-e^{-l\left(\alpha\right)\lambda}}{\lambda}Q_{ij}e^{\lambda x}+o\left(e^{\lambda x}\right),\quad x\rightarrow\infty.
\]
\end{proof}
The proof of Theorem \ref{Main result 2} is analogous to the proof
of Theorem \ref{main result 1} described above. Lemma \ref{lem:lambda for with PF eigenval is 1}
is replaced by Lemma \ref{probability dominant eigenvalue} and as
mentioned above, Lemmas analogous to Lemmas \ref{generating function convergence and formula}\textendash \ref{residue at lambda}
hold when replacing $M$ with the graph probability matrix function
$N$.

\section{Applications\label{sec:Examples-and-Applications}}

\subsection{Summation over regions of Pascal triangle}

The well known triangular array of binomial coefficients contains
many patterns of numbers and properties of combinatorial interest.
It is a straightforward observation that summation of the binomial
coefficients in the triangle $OBA$ of sides $OA=\frac{x}{a}$ and
$OB=\frac{x}{b}$ (see Figure \ref{fig:pascal}) is equivalent to
counting paths of length at most $x$ in a graph with a single vertex
and two loops of lengths $a$ and $b$.

\begin{figure}[h]
\includegraphics[scale=0.7]{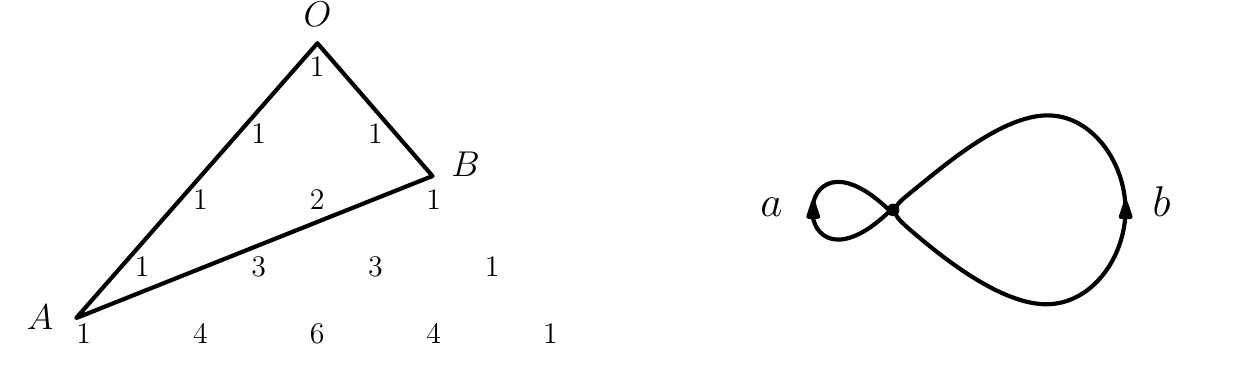}\caption{\label{fig:pascal}A region in Pascal's triangle and the associated
graph.}
\end{figure}

This easily generalizes to Pascal pyramids of higher dimension and
weighted graphs with a singe vertex and several loops, and it would
be interesting to understand the full correspondence between weighted
graphs and regions in Pascal pyramids.

\subsection{Multiscale Substitution Schemes\label{subsec:Multiscale-substitution}}

A tile $T$ in $\mathbb{R}^{d}$ is a Lebegue measurable bounded set with positive measure and boundary of measure zero. Consider a finite set of tiles $\mathcal{F}=\left\{ \mathcal{T}_{1},\ldots,\mathcal{T}_{n}\right\} $
in $\mathbb{R}^{d}$ which we call prototiles, and assume for simplicity $\vol\mathcal{T}_{i}=1$.
A tile $T$ is said to be of type $i$ if it maps to $\mathcal{T}_{i}$
by a similarity map. A multiscale substitution scheme on $\mathcal{F}$ is a set of substitution rules on elements of $\mathcal{F}$ prescribing a tiling of each
prototile by finitely many rescaled copies of tiles of types appearing in $\mathcal{F}$.

A multiscale substitution scheme can be modeled using a directed weighted
graph $G$ with a vertex set indexed by elements of $\mathcal{F}$ and an edge set defined by the substitution rule: if the tiling
of $\mathcal{T}_{i}$ includes a tile of type $j$, that is a copy
of $\alpha\mathcal{T}_{j}$ with $0<\alpha<1$, then $G$ admits
a directed edge of length $a=-\log\alpha$ connecting vertex $i$
to $j$. A multiscale substitution scheme is called irreducible if $\text{\ensuremath{G}}$
is strongly connected, and incommensurable if $G$ is incommensurable.
An example of an incommensurable multiscale substitution scheme on a single
prototile with scales $\alpha_{1}=\frac{1}{3}$ and $\alpha_{2}=\frac{2}{3}$
is shown in Figure \ref{fig:A-multiscale-substitution}.\\

\begin{figure}[h]
\includegraphics[scale=0.7]{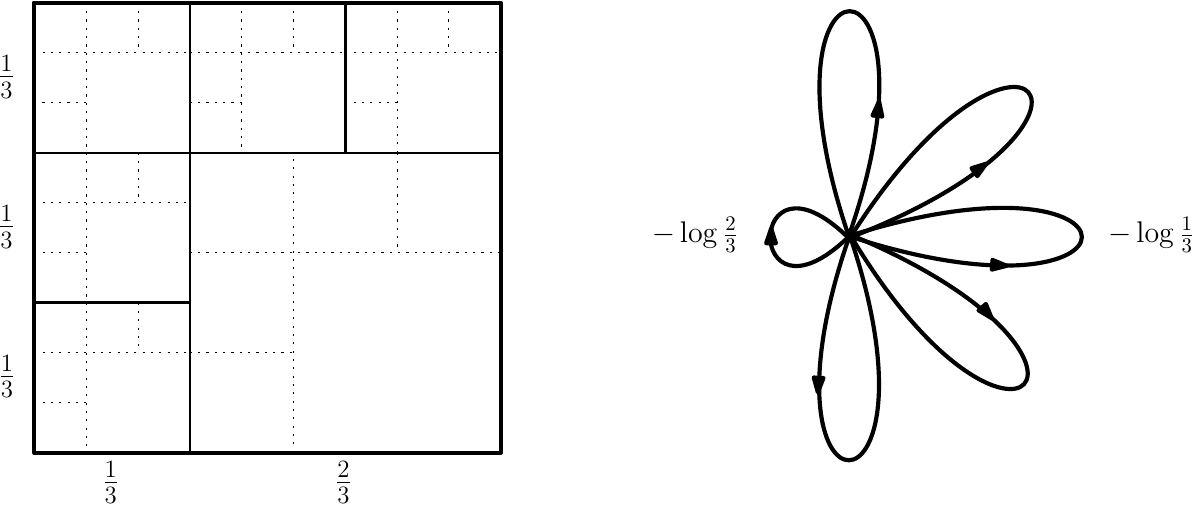}\caption{\label{fig:A-multiscale-substitution}A multiscale substitution scheme
and the associated graph.}
\end{figure}

Observe that if $G$ is a graph associated with a $d$ dimensional
multiscale substitution scheme, then $\lambda=d$ and the Perron-Frobenius
eigenvector of $M\left(d\right)$ corresponding to $\mu=1$ can be
chosen to be $v=\left(1,\ldots,1\right)$. These properties enable
us to address questions concerning the geometrical objects described
bellow.

\subsubsection*{Kakutani Splitting Procedure}

Consider the unit interval $I=\left[0,1\right]$ and some $\alpha\in\left(0,1\right)$.
Kakutani introduced the following splitting procedure which generates
a sequence of partitions of $I$ which is known as the $\alpha$-Kakutani's sequence
of partitions (see \cite{key-3213}). Begin with $\pi_{0}=I$ the
trivial partition of $I$, and define $\pi_{1}$ to be the partition
of $I$ one gets after splitting $I$ into two intervals of lengths
$\alpha$ and $1-\alpha$. Assume that the partition $\pi_{n}$ is
defined, then $\pi_{n+1}$ is the partition of $I$ one gets from
$\pi_{n}$ after splitting the interval of maximal length in $\pi_{n}$
into two parts, proportional to $\alpha$ and $1-\alpha$.

For example, the first few $\alpha$-Kakutani partitions of the unit interval
with $\alpha=\frac{1}{3}$ are shown in Figure \ref{fig:Kakutani's-parition-with},
together with the associated graph. The dashed lines represent intervals
of maximal length in each partition.

\begin{figure}[h]
\includegraphics[scale=0.7]{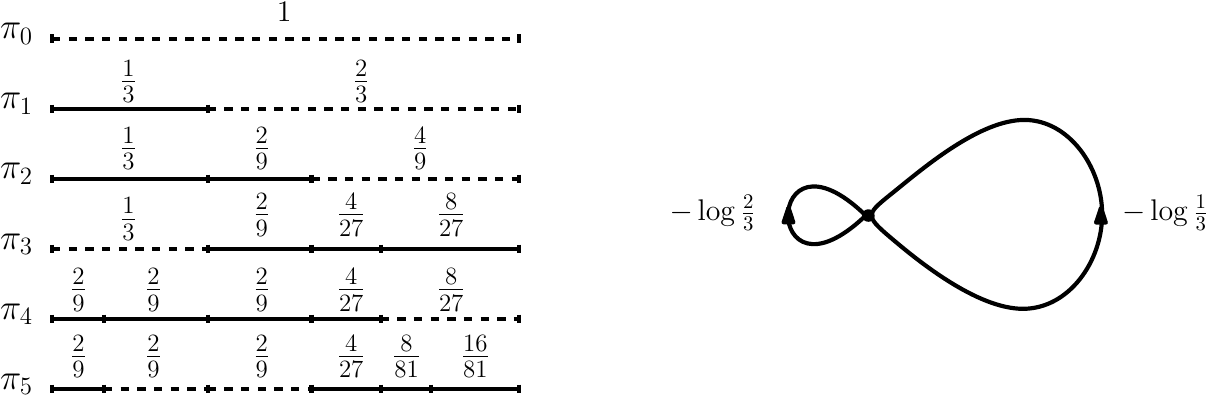}\caption{\label{fig:Kakutani's-parition-with}The $\frac{1}{3}$-Kakutani sequence of partitions and the associated graph.}
\end{figure}

We say that a sequence $\pi_{n}$ of partitions of $I$ is uniformly
distributed if for any continuous function $f$ on $I$ we have
\[
\lim_{n\rightarrow\infty}\frac{1}{k\left(n\right)}\sum_{i=1}^{k\left(n\right)}f\left(t_{i}^{\left(n\right)}\right)=\int\limits _{I}f\left(t\right)dt
\]
where $k\left(n\right)$ is the number of intervals in the partition
$\pi_{n}$, and $t_{i}^{\left(n\right)}$ is the right endpoint of
the $i$th interval in the partition $\pi_{n}$. The following result
is due to Kakutani.
\begin{thm*}
For any $\alpha\in\left(0,1\right)$, the  $\alpha$-Kakutani sequence
of partitions of $I$ is uniformly distributed.
\end{thm*}
Kakutani's splitting procedure is generalized in various ways, see for example \cite{key-14351451} and \cite{key-21413}. An additional generalization comes from multiscale substitution schemes, where we begin with an initial tile $\mathcal{T}_{1}$, and define a sequence of partitions of $\mathcal{T}_{1}$ using the substitution rule applied at each stage to tiles of maximal volume. In fact, the example given above of Kakutani's original procedure, can be considered as generated by a multiscale
substitution scheme in $\mathbb{R}^{1}$ with $\mathcal{F}=\left\{ I\right\} $
the unit interval and $\alpha_{1}=\frac{1}{3}$ and $\alpha_{2}=\frac{2}{3}$.

Using Theorem \ref{main result 1}, it is shown in \cite{key-3154145}
that Kakutani splitting procedures which correspond to incommensurable
multiscale substitution schemes are uniformly distributed.

\subsubsection*{Multiscale Substitution Tilings of Euclidean Spaces}

A multiscale substitution scheme in $\mathbb{R}^{d}$ can be used to
generate a tiling of the entire space. We define a sequence of tilings
of finite regions of $\mathbb{R}^{d}$ which depends on a continuous
time parameter $t$ in the following way: At $t=0$ apply the substitution
rule on an initial tile $\mathcal{T}_{1}$, and inflate the resulting
patch of tiles at a constant speed. Any tile which reaches volume
$1$ is then substituted as dictated by the multiscale substitution
rule, and so on. An appropriate compact topology defined on closed
subsets of the space allows us to take limits of sequences of these
partial tilings, and these limits define tilings of $\mathbb{R}^{d}$.
The generalization of the pinwheel tiling which is presented in \cite{key-8}
can be regarded as a multiscale substitution tiling.

Although there is no uniqueness in the construction of tilings using
multiscale substitutions, all tilings defined this way share various
properties which can be analyzed using the multiscale substitution scheme
itself and the underlying weighted graph. For example, tilings which
are associated with incommensurable multiscale substitution schemes are of
infinite local complexity, unlike classical substitution tilings or
tilings defined using cut-and-project constructions (for more on tilings
and mathematical models of quasicrystals see \cite{key-17367436743}).
Our Theorem \ref{main result 1} may be used to study various statistics
of these tilings, see \cite{key-41234124} for more details, and \cite{key-15362562362356}
and \cite{So2} for relevant results concerning classical substitution
tilings.

\subsection{Physics Applications}

The propagation of radiation pulses through networks of wave-guides
requires for its study the full theory of wave dynamics, where interference
effects play an important role (see \cite{key-234554} and references
therein). However, under certain circumstances the interference effects
can be neglected, which opens the possibility to study this system
within a classical dynamics setting: The network is modeled by a metric,
directed graph $G$, where for any directed edge $\alpha$ connecting
vertex $u$ to vertex $v$, there exists a \textquotedbl{}reverse\textquotedbl{}
edge $\widehat{\alpha}$ from $v$ to $u$ of the same length. The
vertices correspond to the junctions in the network where wave-guides
are connected. In the classical model, one considers a point mass
moving at a unit speed along a directed edge $\alpha$, towards the
vertex $v$. Reaching $v$, the point mass is scattered into any of
the edges $\alpha'$ which emanate from $v$ where it continues to
move with unit speed. The probabilities to make the transitions from
$\alpha$ to $\alpha'$ are prescribed by the properties of the connectors
in the wave-guide network.

The network is connected to the outside world through leads which
are coupled to a subset of vertices $\mathcal{H}$. One of the vertices
$s\in\mathcal{H}$ is connected to a radiation source which sends
short pulses to the network at specified times. Another vertex $t\in\mathcal{H}$
is connected to a lead which ends with a detector where the time of
arrival is measured. The radiation which is scattered to the leads
escapes from the network. In the classical model, a particle is injected
to the vertex $s$ at a given time, and once it scattered from $t$
to the lead, its arrival time is measured. Repeating the process one
can obtain the probability distribution of the transition times. This
model can be analyzed within the formalism provided by Theorem \ref{Main result 2}.

The complete wave theory and the derivation of the corresponding classical
model are provided in \cite{key-234554}. This paper also includes
an analysis of a simple network (similar to the one shown here in
Figure \ref{fig:pascal}, with the lead connected at the single vertex)
and the transition time distributions computed both in the wave and
the classical descriptions are compared.


\begin{thebibliography}{BPP}
\bibitem[An]{key-1}Anlage, S., 2016. Private communication.

\bibitem[BG]{key-17367436743}Baake, M. and Grimm, U.,2013. Aperiodic
order. Vol. 1. A mathematical invitation. Encyclopedia of Mathematics
and its Applications, 149. Cambridge University Press.

\bibitem[BPP]{key-1-1}Broise-Alamichel, A., Parkkonen, J. and Paulin,
F., 2016. Equidistribution and counting under equilibrium states in negatively
curved spaces and graphs of groups. Applications to non-Archimedean
Diophantine approximation, arXiv:1612.06717v1.

\bibitem[CT]{key-2} Chernyshev, V. L. and Tolchennikov, A. A., 2011.
The properties of the distribution of Gaussian packets on a spatial
network, Nauka i Obrazovanie 10, 1\textendash 10

\bibitem[CV]{key-14351451}Carbone, I. and Vol{\v c}i{\v c}, A.,
2007. Kakutani\textquoteright s splitting procedure in higher dimension,
Rend. Ist. Matem. Univ. Trieste XXXIX, 1\textendash 8.

\bibitem[FDC]{key-1-2}Frazer, R. A., Duncan W.J. and Collar A.R.,
1955. Elementary Matrices and Some Applications to Dynamics and Differential
Equations, Cambridge University Press.

\bibitem[Ga]{key-3-1}Gantmacher, F.R., 1979. The theory of matrices,
vol. 2, Chelsea, New York, 1959. Mathematical Reviews (MathSciNet):
MR99f, 15001.

\bibitem[GS]{key-3}Gavish, U. and Smilansky, U., 2007. Degeneracies
in the length spectra of metric graphs. Journal of Physics A: Mathematical
and Theoretical, 40(33), p.10009.

\bibitem[Gu]{key-143513125}Guillope, L., 1994. Entropies et spectres,
Osaka J. Math. 31, 247-289.

\bibitem[HJ]{key-3-2}Horn, R.A. and Johnson, C.R., 1985. Matrix analysis.
Cambridge university press.

\bibitem[Ka]{key-3213} Kakutani, S., 1976. A problem on equidistribution
on the unit interval {[}0,1{]}, Measure Theory, Oberwolfach 1975,
vol. 541, pp. 369\textendash 375. Springer LNM.

\bibitem[KS]{key-4}Kenison, G. and Sharp, R., 2017. Orbit counting in conjugacy classes for free groups acting on trees. Journal of Topology and Analysis, 9(04), pp.631-647.

\bibitem[MV]{key-5}Montgomery, H.L. and Vaughan, R.C., 2006. Multiplicative
number theory I: Classical theory (Vol. 97). Cambridge University
Press.

\bibitem[Pa]{key-6}Parry, W., 1983. An analogue of the prime number
theorem for closed orbits of shifts of finite type and their suspensions.
Israel Journal of Mathematics, 45(1), pp.41-52.

\bibitem[PP1]{key-7}Parry, W. and Pollicott, M., 1983. An analogue
of the prime number theorem for closed orbits of Axiom A flows. Annals
of mathematics, pp.573-591.

\bibitem[PP2]{key-213241}Parry, W. and Pollicott, M., 1990. Zeta
functions and the periodic structure of hyperbolic dynamics. Asterisque
187-188.

\bibitem[Sa]{key-8}Sadun, L., 1998. Some generalizations of the pinwheel
tiling. Discrete \& Computational Geometry, 20(1), pp.79-110.

\bibitem[ScS]{key-234554}Schanz, H. and Smilansky U., 2017. Delay
time distribution in the scattering of time narrow wave packets (II)
Quantum Graphs. J. Phys. A: Math. Theor. 51, 075302.

\bibitem[Se]{key-12362536246246}Seneta, E., 1981. Non-negative matrices
and Markov chains. Springer Series in Statistics. Springer, New York.

\bibitem[SmU]{key-345}Smilansky, U., 2017. Delay-time distribution
in the scattering of time-narrow wave packets. (I) J. Phys. A: Math.
Theor. 50, 215301.

\bibitem[SmY]{key-3154145} Smilansky Y., 2018. Uniform distribution
of Kakutani partitions generated by substitution schemes, arXiv:1805.02213
{[}math.DS{]}.

\bibitem[SyS]{key-41234124} Smilansky, Y. and Solomon, Y., Multiscale
substitution tilings, in preparation.

\bibitem[So1]{key-15362562362356}Solomon, Y., 2011. Substitution
tilings and separated nets with similarities to the integer lattice,
Israel J. Math. 181, 445-460.

\bibitem[So2]{So2}Solomon, Y., 2014. A simple condition for bounded
displacement. J. Math. Anal. Appl. 414 , no. 1, 134\textendash 148.

\bibitem[Vi]{key-123}Viète, F., 1646. \textquotedbl{}Opera Mathematica\textquotedbl{}
F. van Schouten (ed.), Leiden.

\bibitem[Vo]{key-21413}Vol{\v c}i{\v c}, A., 2011. A generalization
of Kakutani's splitting procedure. Ann. Mat. Pura. Appl., Vol. 190,
pp. 45-54.
\end{thebibliography}
\end{document}